\documentclass[a4paper, 12pt]{amsart}
\usepackage[latin1]{inputenc}
\usepackage[T1]{fontenc}
\usepackage[english]{babel}
\usepackage{amssymb}
\usepackage{amsmath}
\usepackage{color}
\usepackage{amsthm}
\usepackage{amscd}
\usepackage{amsfonts}
\usepackage{stmaryrd}
\usepackage{pb-diagram}
\usepackage{epic,eepic,epsfig}
\usepackage{a4wide}
\usepackage{nextpage}
\usepackage{fancyhdr}
\newtheorem{prop}{Proposition}[section]

\newtheorem{lem}[prop]{Lemma}
\newtheorem{theo}[prop]{Theorem}

\newtheorem{theorem}[prop]{Theorem}
\newtheorem*{mthm}{Main Theorem}

\theoremstyle{definition}
\newtheorem{rem}[prop]{Remark}

\newtheorem{defi}[prop]{Definition}

\newcommand{\Sp} {\mathop{\mathrm{Sp}}}

\newcommand{\MT} {\mathop{{\mathrm{MT}}}}
\newcommand{\Gal} {\mathop{\mathrm{Gal}}}

\newcommand{\End} {\mathop{\mathrm{End}}}

\newcommand{\GSp} {\mathop{\mathrm{GSp}}}
\newcommand{\SL} {\mathop{\mathrm{SL}}}
\newcommand{\GL} {\mathop{\mathrm{GL}}}

\newcommand{\SO} {\mathop{\mathrm{SO}}}
\newcommand{\GSO} {\mathop{\mathrm{GSO}}}

\newcommand{\Z}{\mathbb{Z}}
\newcommand{\Q}{\mathbb{Q}}
\newcommand{\G}{\mathbb{G}}
\newcommand{\F}{\mathbb{F}}

\newcommand{\C}{\mathbb{C}}

\newcommand{\Hdg} {\mathop{\mathrm{Hdg}}}

\begin{document}

\title{Remarks on a theorem of Pink in presence of bad reduction}
\author{Wojciech Gajda 
and Marc Hindry}
\maketitle

\begin{abstract}
In this note we prove new cases of the Mumford-Tate conjecture  by extending a theorem of Richard Pink for abelian varieties without nontrivial endomorphisms and with bad 
semistable reduction. We use quadratic pairs introduced by J.G.Thompson in the seventies, an important tool in the 
program of classifying all simple finite groups.  Proof of our main result 
applies the classification of the quadratic pairs as described by Premet and Suprunenko.  Along the way we reprove and generalize a theorem of Chris Hall on the image of Tate module representation of abelian variety as above, to all possible values of its toric dimension. 
\end{abstract}

\section{Introduction}

Let $A$ be a simple abelian variety  of dimension $g$ defined over a number field $K$. Suppose that $A/K$ has bad (split) semistable reduction at some place $v$ of the field $K$, denote $\F_v$ the finite residual field.
Let $\mathcal{A}\rightarrow{\rm spec}(\mathcal{O}_K)$ be the N\'eron model; the connected component  of its special fibre lies in an exact sequence of algebraic ${\bf F}_v$-groups:
\begin{equation}
0\longrightarrow T={\Bbb G}_m^s\longrightarrow\mathcal{A}_v^0\longrightarrow B_v\longrightarrow 0,
\end{equation}
where $B_v$ is an abelian variety over $\F_v$.
We call $s=\dim T$ {\it the toric dimension}   of $A$ at $v$.

The Mumford-Tate group $\MT(A)$ is an algebraic $\Q$-group (cf. for example \cite{pink1}, for the definition) defined in terms of the complex structure of $A(\C)$, that comes together with a representation on $V{:=}H^1(A(\C),\Q)$. When the abelian variety is of type  I, II or III  in Albert classification, the Hodge group $\Hdg(A)$ is the derived group of $\MT(A)$ and we have $\MT(A){=}\G_m\Hdg(A)$.
Denote $T_{\ell}(A)$ the Tate module and $V_{\ell}(A){:=}T_{\ell}(A)\otimes_{\Z_{\ell}}\Q_{\ell}$; the central object of our study is the Galois representation $\rho_{\ell}:\Gal(\bar{K}/K)\rightarrow \GL(V_{\ell}(A))$. We denote $G_{\ell}$ the connected component of the Zariski closure of $\rho_{\ell}(\Gal(\bar{K}/K))$ in $\GL(V_{\ell}(A))\cong\GL_{2g,\Q_{\ell}}$, it is an algebraic $\Q_{\ell}$-group and we still denote $\rho_{\ell}$ its representation  on $V_{\ell}$.
Via the comparison isomorphism $V_{\ell}{\cong} V\otimes_{\Q}\Q_{\ell}$ we can compare $\MT(A)\otimes_{\Q}\Q_{\ell}$ and $G_{\ell}$ inside $\GL_{2g,\Q_{\ell}}$. A deep result due to Borovo\"i \cite{bor1}, Deligne \cite{delm} and Piatetski-Shapiro \cite{piat} shows that $G_{\ell}$ is an algebraic subgroup of $\MT(A)\otimes_{\Q}\Q_{\ell}$,  the Mumford-Tate conjecture asserts the equality of these two algebraic groups. We refer the reader to a variety of surveys on the current state of arts of the Mumford-Tate conjecture,  which is available online.  

It is known that the Mumford-Tate conjecture holds for abelian varieties over number fields,  if and only if,  it holds for simple ones cf. \cite{com}.  Further, it is easy to see that replacing $K$ by a finite extension or $A$ by an isogenous abelian variety does not alter the validity of Mumford-Tate conjecture; it is a deeper result (cf. \cite{lp3}, Theorem 4.3) that Mumford-Tate conjecture does not depend on the chosen prime $\ell$. Hence,  without loosing any generality, we can concentrate on simple abelian varieties over $K$, assume semistability and pick any convenient $\ell$.

\begin{mthm}\label{main}
Let $A$ be an abelian variety defined over a number field $K$. Suppose that $\End(A)=\Z$ and there is a place of bad semistable reduction with toric dimension $s$. Then $\MT(A)_{\Q_{\ell}}=\GSp_{2g,\Q_{\ell}}=G_{\ell}$,   and  the Mumford-Tate conjecture holds  for $A,$ except may be in following exceptional cases.
\begin{enumerate}
\item There is an odd integer $r\geq 3$ such that
$g=\frac{1}{2}{2r\choose r}$ and $s={2r-2\choose r-1}$.
\item There is an integer $t\geq 4$ with $t{\equiv} 0$ or $1 \mod 4,$ such that $g=2^t$ and $s=g$ or $g/2$.
\end{enumerate}
\end{mthm}

\noindent We remark that the hypotheses in the theorem are {\it generic} in the sense that a generic abelian variety of dimension $g$ is simple with endomorphism ring $\Z$, whereas abelian varieties with some bad reduction are far more numerous that abelian varieties with (potential) good reduction everywhere, the typical case being for $g=1$ where integers are rarer than rational numbers.
The first instances of the first exceptional case in Main Theorem occur for $(g,s)=(10,6)$, $(84,70)$, $(1716,924)$ etc.The first instances of the second exceptional case occur for $(g,s)=(16,8)$, $(16,16)$, $(32,16)$, $(32,32)$, $(256,128)$, $(256,256)$ etc. 
\bigskip

\noindent
{\sl Comparison with former results}
\medskip

\noindent
Compared to former attemps at the conjecture (which were often based on computation of Lie algebras of algebraic groups,  e.g.,  \cite{bgk, bgk1, bgk2}) our method is different. The novelty of this note approach depends on applying theory of representations of finite groups of Lie type directly to comparison of special fibres of two Chevalley group schemes in question, and then lifting equality to their $\Q_{\ell}-$points (cf.  Lemma \ref{marcnikolay}) in order to confirm the conjecture. Our main theorem shall be compared to {a result of Pink (under the assumption $\End(A){=}\Z$) and to {the following }results of Noot,  and Chris Hall.}

\begin{theorem}\label{pinkmtc}(Pink  \cite{pink1}, Theorem 5.14) Let $A$ be an abelian variety of dimension $g$ with endomorphism ring $\Z$. 
Assume that $g$ is neither
\begin{itemize}
\item half of a $k^{\rm th}$ power for {any} odd $k>1$, nor
\item of the shape $\frac{1}{2}{2m\choose m}$ for {any} odd $m\geq 3$.
\end{itemize}
Then $G_{\ell}{=}\GSp_{2g,\Q_{\ell}}$ and the Mumford-Tate conjecture holds for $A$.
\end{theorem}

\begin{theorem}\label{4noot} {\rm (Noot \cite{noot1})} Let $A$ be an abelian variety of dimension 4 with endomorphism ring $\Bbb Z$. Assume that $A$ is defined over a number field and has a place with bad semistable reduction. { Then $G_{\ell}=\GSp_{8,\Q_{\ell}}$ and the Mumford-Tate conjecture holds for $A$.}
\end{theorem}

\begin{theo}\label{cha}(C. Hall \cite{hall}) Let $A$ be an abelian variety of dimension $g$ with endomorphism ring $\Z$. Assume that $A$ has a place of bad semistable reduction with toric dimension $s=1. $ {Then $G_{\ell}{=}\GSp_{2g,\Q_{\ell}}$ and the Mumford-Tate conjecture holds for $A$.}
\end{theo}

\noindent
Our proof of Main Theorem uses quadratic pairs introduced by J.G.Thompson in \cite{qp}. We apply the classification of the quadratic pairs described by Premet and 
Suprunenko in \cite{premet}.  Note that along the way we reprove and generalize Theorem \ref{cha} to all possible values of $s.$

\section{Quadratic pairs and minuscule representations}
\begin{defi} A {\it quadratic pair} $(G,\rho)$  is formed by a group $G$, a faithful irreducible representation $\rho: G\rightarrow\GL(V)$ such that the subset $\mathcal{Q}:=\{g\in G\setminus\{1\}\;|\; (\rho(g)-1_V)^2=0\}$ is non empty and generates $G$. (This is Thompson's definition \cite{qp} when we assume $G$ finite and $V$ a $\F_p$-vector space; otherwise we understand ``generates" as ``topologically generates"). A quadratic pair is {\it polarised} 
(orthogonal or symplectic) if $G$ leaves invariant a symmetric or antisymmetric non degenerate bilinear form. A quadratic pair is {\it simple} if $G$ is equal to its commutator group and the quotient of $G$ by its centre is simple.
\end{defi}
The main properties of quadratic pairs are a decomposition theorem into (almost) simple blocs and a classification of these simple blocs which are due to Thompson \cite{qp}, see also \cite{ho1,betty} and complemented by Premet-Suprunenko in \cite{premet}.
\begin{prop} Let $(G,\rho)$ be a quadratic pair with $V$ an $\F_p$-vector space.
\begin{enumerate}
\item (Thompson's central product theorem) There exist simple quadratic pairs $(G_i,\rho_i)$ such that $G{=}G_1\cdots G_t$, and $(G,\rho)$ is isomorphic to $\rho_1\otimes\dots\otimes\rho_t:G_1\cdots G_t\rightarrow \GL(V_1\otimes{\dots}\otimes V_t)$. Further, if the pair $(G,\rho)$ is polarised, then each $(G_i,\rho_i)$ is polarised.
\item (Thompson, Premet-Suprunenko) A simple quadratic pair belongs to the following list which is given indicating the root system associated to a Chevalley group and a short (generally accepted) name for the representation.  
\begin{itemize}
\item $(A_n,\wedge^jStd)$
\item $(B_n,Std)$ or $(B_n, Spin)$
\item $(C_n,\mathcal{W}_j)$ with a subrepresentation $\mathcal{W}_j\subseteq \wedge^jStd$
\item $(D_n,Std)$ or $(D_n,Spin^{\pm})$
\item Explicitly given finite set of representations of exceptional groups of type $E_6$, $E_7$, $F_4$ and $G_2$.
\end{itemize}
\end{enumerate}
\end{prop} 

We will require the following elementary lemma.
An element $g\in\End(V)$ is {\sl $k$-unipotent} if $(g-1)^k=0$ but  $(g-1)^{k-1}\not=0$; for example a quadratic element is, by definition, $2$-unipotent.
 \begin{lem}\label{kunip} Let $g_1$  (resp. $g_2$) be a $k_1$-unipotent element in $\End(V_1)$ (resp. a $k_2$-unipotent element in $\End(V_2)$), then $g_1\otimes g_2$ is a $(k_1+k_2-1)$-unipotent element in $\End(V_1\otimes V_2)$. In particular,  quadratic elements of a quadratic pair $(G_1\cdots G_t,\rho_1\otimes\cdots\otimes\rho_t)$ lie in some $G_i$.
 \end{lem}
 \begin{proof} See \cite{paug},  Lemma 2.2.1.
 \end{proof}
 
 Our typical example of quadratic pair is obtained by considering the Galois representation on $A[\ell]\cong(\Z/\ell\Z)^{2g}$ and restricting it to the subgroup $R$ generated by inertia and their conjugates. One sees that this representation is still irreducible for $\ell$ large enough. Minuscule representation appear when considering the $\ell$-adic representation.
 
The link between quadratic pairs and minuscule representations is suggested in \cite{serreht} (see the remark after Corollaire, page 180), where it is observed that root elements of a (semi-simple) group act quadratically. To pass from representations mod $\ell$ to representations over $\Z_{\ell}$  or $\Q_{\ell}$ we'll use the following elementary lemma.
 \begin{lem}\label{marcnikolay} Let $G$ be a smooth algebraic subgroup of $\GL_N$ over $\Z_{\ell}$ and $H$ a closed subgroup of $G(\Z_{\ell})$ such that $\pi(H)=G(\F_{\ell})$,  where $\pi$ is the reduction mod $\ell$ map. Assume $\ell\geq 5$ and the Lie algebra of $G_{\F_{\ell}}$ is spanned by quadratic elements. Then $H=G(\Z_{\ell})$.
\end{lem}
\begin{proof} This is a special case of Lemme 2.6 in \cite{hr2}.
\end{proof}
The lemma applies to $SL_m$, to $\Sp_{2m}$ {\it loc.cit.} and more generally to any quadratic representation ! 
 Thus the property of being spanned by null square matrices is one of the characterisations of minuscule representation, see for example \cite{paug} (definition 2.1.3), \cite{serreht} (Lemme 5 and Corollaire after Proposition 5).

Below we recall the table of minuscule weights of height one (reference: \cite{serreht}) giving the type of root system (A,B,C,D), the label of the highest weight  describing the representation (labeled as in {\it loc. cit.}),  a short name for the representation, the dimension of the  representation and finally the ``sign" of the representation: $+1$ for an orthogonal representation, $-1$ for a symplectic representation, and $0$ for a non auto-dual representation.
\medskip

\begin{center} \it Table of minuscule representations
\end{center}
 $$\begin{tabular}{|c|c|c|c|c|}
 \hline
 Lie type & weight &representation& dimension & sign\\
 \hline
 $A_n$\; ($n\geq$ 1)& $\varpi_j\; (1\leq j\leq n)$&$\Lambda^j({\rm Std})$&${n+1\choose  j}$&$\begin{matrix} (-1)^j&{\rm if}\;\; n=2j-1\cr 0& {\rm else}\end{matrix}$\\
 \hline
  $B_n$\; ($n\geq$ 2)& $\varpi_n$&${\rm Spin}$&$2^n$&$\begin{matrix} +1&{\rm if}\;\; n\equiv 0, 3\mod 4 \cr -1&{\rm if}\; \equiv 1,2\mod 4 \end{matrix}$\\
 \hline
  $C_n$\; ($n\geq$ 2)& $\varpi_1$&${\rm Std}$&$2n$&$-1$\\
 \hline
  $D_n$\; ($n\geq$ 3)& $\begin{matrix}\varpi_1\cr \cr\varpi_{n-1}\cr\varpi_n\cr\end{matrix}$ &$\begin{matrix}{\rm Std}\cr \cr{\rm Spin}^-\cr{\rm Spin}^+\cr\end{matrix}$ &$\begin{matrix}2n\cr \cr 2^{n-1}\cr 2^{n-1}\cr\end{matrix}$ &$\begin{matrix} +1\cr \cr
   \left\{\begin{matrix} +1&{\rm if}\;\; n\equiv 0\mod 4 \cr -1&{\rm if}\; n\equiv 2\mod 4 \cr 0&{\rm if}\; n \;{\rm odd}\cr\end{matrix}\right.\cr\end{matrix}$\\
 \hline
 \end{tabular}
 $$
\bigskip\bigskip

\section{Proof of Main Theorem}

Existence of a place of bad reduction with toric dimension $s$ implies that a topological generator of (the image of) inertia provides a {\it quadratic} element $g$ with {\it drop} equal to $s$.  Recall that a quadratic element of $(G,\rho)$ is an element such that $(\rho(g)-1)^2=0$ and the drop is the dimension of the image of  $\rho(g)-1$.  
 Let us denote by $G'_{\ell}/\Q_{\ell}$ the derived group of the algebraic monodromy group $G_{\ell}:=\overline{\rho_{\ell}(\Gal(\bar{K}/K))}^{\rm Zar}$ which, without loss of generality,  can be assumed connected.  For   simplicity we treat the case $\End(A){=}\Z$ although our methods apply to more general situations (cf.  Section 4). 
The algebraic group $G'_{\ell}$ and the natural representation $\rho_{\ell}:G'_{\ell}\rightarrow\GL(T_{\ell}(A))\subset\GL_{2g,\Q_{\ell}}$  have very special features.
\begin{itemize}
\item The representation is faithful symplectic.
\item (Faltings, \cite{falt}) The $\ell$-adic representation is absolutely  irreducible; further, for $\ell$ large enough, the representation mod $\ell$ is still irreducible.
\item (Pink, \cite{pink1}) The representation is a tensor product of representations $\rho_i:G_i\rightarrow \GL(V_i)$, where $G_i$ are absolutely irreducible groups, $G'_{\ell}=G_1\cdots G_s$ (an almost direct product), $V_{\ell}\cong V_1\otimes\cdots\otimes V_s$ and each of these factors is a classical minuscule symplectic representation.  Further,  the decomposition is isotypical with a transitive action of the Galois group.
\end{itemize}

\noindent
We distinguish two cases, using quadratic pairs in the first case and minuscule representations for the second: 
\begin{itemize}
\item (A)  $s>1$ (hence odd and $\geq 3$).
\item (B) $s=1$ (i.e. the group $G'_{\ell}$ is absolutely irreducible).
\end{itemize}

\noindent
{\bf Case (A)} Assumption on $s$ implies (denoting $m{=}\dim V_i$) that $2g{=}m^s$ as in \cite{pink1}, but  we will show that, when there is a place of bad reduction, this case does not occur.
We choose $\ell$ large enough so that the representation modulo $\ell$, which we denote $\bar{\rho}_{\ell}:\Gal(\bar{K}/K)\rightarrow \GL(A[\ell])\cong\GL_{2g}(\F_{\ell})$, is still irreducible. Let $\bar{R}$ be the subgroup generated by the image 
of the inertia group at the place of bad reduction and its conjugates. We first claim that the restriction of $\bar{\rho}_{\ell}$  to $\bar{R}$ is still irreducible. Indeed by Clifford's classical theorem, the representation is induced from an irreducible subrepresentation $W$, but $(\bar{R},W)$ forms a quadratic pair and then $\bar{R}$ is a finite group of Lie type and has no normal subgroups of small index, at least if $\ell$ is chosen sufficiently large, therefore $V{=}W$. Next, $\bar{R}$ is generated by the conjugacy class of a quadratic element, therefore it cannot be an almost direct product $G_1\cdots G_s$ with $s{>}1$.
 The argument for this is that a quadratic element $g$ has to occur in only one factor, i.e.,  $g{=}(1,\dots, g_i,\dots, 1)\in G_1\cdots G_s$. We now lift $\bar{R}$ to $R\subset G'_{\ell}$, where $R$ denotes the subgroup generated by inertia in the $\ell$-adic representation. Since the Galois group permutes the factors, the inertia cannot sit in only one factor (cf.  Lemma \ref{kunip}) and we conclude that $s{=}1$. Note that the case of $\dim A{=}4$ was initially treated by Noot in Prop. 2.1 and Cor. 2.2, \cite{noot1}. 
\medskip

\noindent
{\bf Case (B). }The list of symplectic minuscule representations of dimension $2g$ of a classical group is short; see the table in previous section; we denote $(G_1,\rho_1)\sim(G_2,\rho_2)$ when the two simple algebraic groups $G_i$ have the same root system and $\rho_i$ correspond to the same weight.  \begin{enumerate}
\item $(G_{\ell}',\rho)\sim (\Sp_{2g},Std)$.
\item $(G_{\ell}',\rho)\sim (\SL_{2r},\wedge^rStd)$, with odd $r\geq 3$ and $2g={2r \choose r}$.
\item $(G_{\ell}',\rho)\sim (\SO_{2r+1},Spin)$, with $2g=2^r$ and $r\equiv 1, 2\mod 4$.
\item $(G_{\ell}',\rho)\sim (\SO_{2r},Spin^{\pm})$, with $2g=2^{r-1}$ and $r\equiv 2\mod 4$.
\end{enumerate}
In the first case we conclude $G'_{\ell}=\Sp_{2g,\Q_{\ell}}=\Hdg(A)_{\Q_{\ell}}$. 
In the three last cases we use the computation by Premet-Suprunenko of the drop of a quadratic element (cf.  \cite{premet}, Section 3).    
More precisely,  in the second case we see that $\dim A=\frac{1}{2}{2r \choose r}$ is very specific.  Further,  the drop of a non-trivial quadratic element in this representation is $s:={2r-2\choose r-1}$ (see \cite{premet},  Lemma 18 and the calculation at p.78,  beginning of \S 2), therefore this case can only happen if the toric dimension at a bad place is equal to $s$.  In the third case we see that $\dim A=2^{r-1}$ is very specific. Further, the drop of a non-trivial  quadratic element in this representation is $s:=2^{r-1}$ or $2^{r-2}$ (see \cite{premet} Lemma 20), therefore this case can only happen if the toric dimension at a bad place is equal to $s=g$ or $g/2$.  For the spin representation to be symplectic, we need $r\equiv 1,2\mod 4$.  In the fourth case we see that $\dim A=2^{r-2}$ is very specific.  Further, the drop of a non-trivial quadratic element in this representation is $s:=2^{r-2}$ or $2^{r-3}$ (see \cite{premet} Lemma 21 and Note 2, at p.86),  therefore this case can only happen, if the toric dimension at a bad place is equal to $s=g$ or $g/2$.  For the half-spin representation to be symplectic, we need $r\equiv 2\mod 4$. This completes the proof of Main Theorem.\endproof

\begin{rem}\label{remfinal1} Our proof shows that inertia subgroups at bad places generate the derived group of the image of Galois.  Indeed the above arguments show that the algebraic subgroup $R$ generated by the inertia subgroup and their conjugates is a non trivial normal subgroup and the only such subgroup of $\Sp_{2g}$ is $\Sp_{2g}$ itself.
\end{rem}

\section{Conclusion}

As a final remark we add the observation that the same method can handle the case when  the endomorphism algebra $\End^0(A):=\End(A)\otimes\Q$ is a quaternion algebra with center $\Q$. However,  note that the combinatorics,  when the center is larger, is much more involved.  We plan to address this issue in a future work. 

\begin{theorem}\label{last}
Let $A$ be an abelian variety defined over a number field $K$. Suppose that $\End^0(A):=\End(A)\otimes\Q$ is a quaternion algebra with centre $\Q$ and there is a place of bad reduction with toric dimension $s$ (note $s$ is always even). Then, 
\begin{itemize}
\item 
when the algebra $D$ is indefinite, i.e., $A$ has type II, we have $\MT(A)_{\Q_{\ell}}{=}\GSp_{g,\Q_{\ell}}{=}G_{\ell}$ and the Mumford-Tate conjecture holds for $A,$ except may be in following exceptional cases:
\begin{enumerate}
\item there is an odd integer $r\geq 3$ such that
$g={2r\choose r}$ and $s=2{2r-2\choose r-1}$,
\item there is an integer $t{\geq 5}$ 
with $t{\equiv} 1$ or $2\mod 4$ such that $g{=}2^t$ and  $s{=}g$ or $g/2,$
\end{enumerate}
\item when the algebra $D$ is definite, i.e.,  $A$ has type III,  we have $\MT(A)_{\Q_{\ell}}{=}\GSO_{g,\Q_{\ell}}{=}G_{\ell}$ and  the Mumford-Tate conjecture holds for $A,$ except may be in following exceptional cases:
\begin{enumerate}
\item there is an even integer $r\geq 2$ such that
$g={2r\choose r}$ and $s=2{2r-2\choose r-1}$,
\item there is an integer $t{\geq} 4$ with $t{\equiv} 0$ or $3\mod 4$   such that $g{=}2^t$ and  $s{=}g$ or $g/2$.
\end{enumerate}
\end{itemize}
\end{theorem}

\noindent
The extension of Theorem \ref{pinkmtc} (and Theorem \ref{last}) when $s=2$ to abelian varieties of type II  is worked out in \cite{hr2}, Th\'eor\`emes 10.6 and  10.7. 
\bigskip

\noindent\textsc{\bf Acknowledgements.} Both authors thank professor Ronald Solomon for his guidance through the literature on quadratic pairs.
 The authors were partially supported by a research grant UMO-2018/31/B/ST1/01474 of the National Centre of Sciences of Poland.

\newpage

\small
{\sc Wojciech Gajda\\
Faculty of Mathematics and Computer Science\\
Adam Mickiewicz University\\
Uniwersytetu Pozna\'nskiego 4\\
61-614 Pozna\'{n}, Poland}\\
E-mail adress: \texttt{\small gajda@amu.edu.pl}
\bigskip\bigskip\bigskip

\small
{\sc Marc Hindry\\
Institut de Mathematiques de Jussieu - Paris rive gauche (Imj-Prg),\\
Universit\'e Paris Cit\'e (Paris  Diderot),\\
B\^atiment Sophie Germain,\\
Boite Courrier 7012,\\
8 place Aur\'elie Nemours,  \\
F-75205 Paris CEDEX 13, France.}\\
E-mail address: \texttt{\small marc.hindry@imj-prg.fr}
\end{document}